\documentclass[11pt]{amsart}
\usepackage{latexsym, amsmath, amssymb, amsthm, mathrsfs}
\usepackage[alwaysadjust]{paralist}
\usepackage{caption}
\usepackage{subcaption}
\usepackage[T1]{fontenc}
\usepackage{lmodern}
\usepackage[backref,colorlinks=true,linkcolor=blue,urlcolor=blue, citecolor=blue]%
  {hyperref}
\usepackage{amsrefs}               %\usepackage[alphabetic]{amsrefs}
\usepackage[all]{xy}
\usepackage[T1]{fontenc}
\usepackage{mathptmx}
\usepackage{microtype}
\usepackage{framed}
\usepackage{tikz}
\usepackage{graphicx}
\usepackage[alwaysadjust]{paralist}
\makeatletter
\providecommand\@enum@widestlabel{7}
\makeatother

\usepackage[centering, includeheadfoot, hmargin=0.8in, vmargin=0.8in,
  headheight=30.4pt]{geometry}
  
\newtheorem{lemma}{Lemma}[section]
\newtheorem{theorem}[lemma]{Theorem}
\newtheorem*{theorem*}{Theorem}
\newtheorem{corollary}[lemma]{Corollary}
\newtheorem{proposition}[lemma]{Proposition}
\newtheorem{conjecture}[lemma]{Conjecture}

\theoremstyle{plain}
\newtheorem*{theoremA}{Theorem A}
\newtheorem*{theoremB}{Theorem B}

\theoremstyle{definition}
\newtheorem{definition}[lemma]{Definition}
\newtheorem{remark}[lemma]{Remark}

\theoremstyle{remark}

\renewcommand{\theequation}%
{\arabic{section}.\arabic{lemma}.\arabic{equation}}
\newcommand{\CC}{\ensuremath{\mathbb{C}}}

\newcommand{\PP}{\ensuremath{\mathbb{P}}}

\newcommand{\QQ}{\ensuremath{\mathbb{Q}}}

\newcommand{\ZZ}{\ensuremath{\mathbb{Z}}}
\newcommand{\sI}{\ensuremath{\kern -1pt \mathscr{I}\kern -2pt}}
\newcommand{\sJ}{\ensuremath{\kern -2pt \mathscr{J}\kern -2pt}}
\newcommand{\sO}{\ensuremath{\mathscr{O}}}
\newcommand{\sOX}{\ensuremath{\mathscr{O}^{}_{\! X}}}

\renewcommand{\geq}{\geqslant}
\renewcommand{\leq}{\leqslant}

\DeclareMathOperator{\mult}{mult}

\DeclareMathOperator{\Pic}{Pic}

\DeclareMathOperator{\Sym}{Sym}

\DeclareMathOperator{\Cliff}{Cliff}

\newcommand{\equ}{\ensuremath{\,=\,}}
\newcommand{\dgeq}{\ensuremath{\,\geq\,}}
\newcommand{\dleq}{\ensuremath{\,\leq\,}}
\newcommand{\deq}{\ensuremath{\stackrel{\textrm{def}}{=}}}

\usepackage{ulem}

\begin{document}

\title{Higher syzygies on surfaces with numerically trivial canonical bundle}
\author{Daniele Agostini}
\author{Alex K\"uronya}
\author{Victor Lozovanu}
\address{Daniele Agostini, Humboldt Universit\"at zu Berlin, Institut f\"ur Mathematik, Unter den Linden 6, D-10099 Berlin, Germany}
\email{\tt daniele.agostini@math.hu-berlin.de}

\address{Alex K\"uronya, Johann-Wolfgang-Goethe Universit\"at Frankfurt, Institut f\"ur Mathematik, Robert-Mayer-Stra\ss e 6-10., D-60325	Frankfurt am Main, Germany}
\address{Budapest University of Technology and Economics, Department of Algebra, Egry J\'ozsef u. 1., H-1111 Budapest, Hungary}
\email{{\tt kuronya@math.uni-frankfurt.de}}

\address{Victor Lozovanu, Institut f\"ur Algebraische Geometrie, Gottfried-Wilhelm-Leibniz-Universit\"at Hannover,
	Welfengarten 1, D-30167 Hannover, Germany}
\email{\tt victor.lozovanu@gmail.com}

\maketitle

\section{Introduction}

The aim of this paper is to understand higher syzygies of polarized surfaces with trivial canonical bundles. More concretely, for a polarized surface 
$(X,L)$ with $K_X=0$ over the complex numbers,  we establish a tight connection between property $(N_p)$ for $(X,L)$ and the (non)existence of certain forbidden subvarieties of $X$. 
Our results unify  the main statement of \cite{KL15}, and the classical theorem of Saint-Donat \cite{SD} on complete linear series on K3 surfaces, which  hints at 
a beautiful interplay between the  local and global geometry of the surface.

% \begin{theorem}[Saint-Donat \cite{SD}]\label{thm:saintdonat}
% Let $L$ be an ample  line bundle on a a smooth $K3$ surface $X$. Then 
% \begin{enumerate}
%  \item If $(L^2)\geq 4$, then the linear series $|L|$ fails to define a projectively normal embedding if and only if there is no smooth elliptic curve $F\subseteq X$ 
%  passing through a very general point on $X$ with $1\leq (L\cdot F)\leq 2$, or $L=2M$ for a line bundle $M$ such that $(M^2)=2$.
%  \item If $(L^2)\geq 8$, then the linear series $|L|$ defines a projectively normal embedding whose homogeneous ideals is generated by quadrics if and only if there is no smooth elliptic curve $F\subseteq X$ passing through a very general point on $X$ with  $1\leq (L\cdot F)\leq 3$, or $L=2M$ for a line bundle $M$ such that $(M^2)=2$. 
% \end{enumerate}
% \end{theorem}
 
In the current context syzygies of projective embeddings have entered the stage after the seminal work of Green \cite{G84}, and have become an active area of 
research ever since. From a geometric point of view, it has been observed  that projective normality of an embedding and the conditon of the homogeneous ideal of the embedding being 
generated by quadrics can be seen as  the first elements of   an increasing sequence of positivity  properties, namely the cases $p=0$ and $p=1$ for property $(N_p)$. 
The arising conjectures of Green and Lazarsfeld energized the subject, which then has become  an important topic in algebraic geometry  to this day. For some of the highlights we refer to 
\cites{EL1,V02,V05} and the book \cite{AN10}. The last years have witnessed a new line of research which focuses on the  asymptotic behaviour of Koszul cohomology groups 
(see \cite{EL16} and the references therein). 

Our guiding principle is that in the case of a polarized surface $(X,L)$ with $K_X\equiv 0$ and $(L^2)\gg 0$, one can  expect an equivalence between  the vanishing 
of certain Koszul cohomology groups of low degree for $(X,L)$, and the non-existence of elliptic curves of low $L$-degree on $X$ passing through a very general point of $X$. 
In addition, we will link the equivalence  to the local positivity of $(X,L)$ at a very general point, which we will measure in terms of Seshadri constants. 

With this in mind, the main result of this article is the following. 

\begin{theoremA}\label{thm:A}
 Let $p$ be a natural number, $X$ a smooth projective surface over the complex numbers with trivial canonical bundle, and $L$ an ample line bundle on $X$ with $(L^2)\geq 5\cdot (p+2)^2$. 
 Then the following are equivalent.
\begin{enumerate}
\item $(X,L)$ satisfies property $(N_p)$.
\item There is no smooth elliptic curve $F\subseteq X$ (through a very general point of $X$) with $1\leq (L\cdot F)\leq p+2$ and $(F^2)=0$.
\item The Seshadri constant at a very general point $x$ satisfies $\epsilon(L;x)>p+2$.
\end{enumerate}
\end{theoremA}

\begin{remark}
The parentheses in condition (2) amount to the fact that if there is a single smooth elliptic curve with $(L\cdot F) \leq p+2$, then there is one passing through a very general point with the same property. This is clear for abelian surfaces since we can move the curve with translations, and for K3 surfaces it is a classic application of Riemann-Roch: see for example \cite{beauville}*{Proposition VIII.13}.
\end{remark}

We will prove Theorem A in Section 3, where we will observe in addition  that the constant $5$ can be much improved in the $K3$ case. By setting $p=0$ and $p=1$, we will recover 
Saint-Donat's results when $(L^2)\geq 10$. Beside the by now classical theory of Koszul cohomology, our argument relies on  recent work of Aprodu and Farkas \cite{AF11}, 
where  Green's conjecture for smooth curves on K3 surfaces is verified, and earlier work  of Serrano \cite{Se87}. 

Going further, we establish an similar result for Enriques surfaces; this time however our result is conditional on the Green--Lazarsfeld secant conjecture for curves (for a precise 
statement see Conjecture~\ref{conj:GL secant}). 
 
\begin{theoremB}\label{thm:B}
Let $p$ be a natural number, $X$ an Enriques surface over the complex numbers,  and $L$ an ample and globally generated line bundle on $X$ such that $(L^2)>4(p+2)^2$.
Assume  that the Green--Lazarsfeld secant conjecture  holds for a general curve in  $|L|$. Then the following are equivalent
\begin{enumerate}
\item $L$ has property $N_p$.
\item there is no effective and reduced divisor $F$ with $(F^2)=0$ such that $1\leq (E\cdot L)\leq p+2$.
\end{enumerate}  
\end{theoremB} 
 
We prove Theorem B as Theorem~\ref{thm:Enriques} in Section 5  by reducing to the one-dimensional case using Koszul cohomology. Our argument is  supported by results of Knutsen and Lopez~\cite{knutsenlopez} on linear series on Enriques surfaces.

\begin{remark}
We believe that the lower bound  $(L^2)\geq 5(p+2)^2$ in Theorem A can be improved to $(L^2) > 4(p+2)^2$. This is true for K3 surfaces as a consequence of the more precise estimates of Theorem \ref{thm:gle}, so that the question remains open for abelian surfaces. This would also be the optimal bound for abelian surfaces: if $p=0$, then by work of Barth  \cite{barth}, a general 
polarized abelian surface $(X,L)$  of type $(2,4)$ (in particular with $(L^2)=16$) does not have property $N_0$, at the same time  such a surface does not contain  elliptic curves.
\end{remark}

The more accurate results  of Theorem~\ref{thm:gle} and Proposition~\ref{prop:seshadri}  establish Theorem~A for polarized K3 surfaces. As an easy consequence of Theorem~\ref{thm:gle} we also get a characterization of property $(N_p)$ for ample line bundles of type $L^{\otimes m}$, with $m\geq p$, along the lines of Mukai's conjecture. The case of abelian surfaces follows from \cite{KL15}*{Theorem 1.1} and Proposition~\ref{prop:seshadri}.

About the organization of the article: In Section 2 we establish the necessary preliminaries,  collect a few useful statements about Koszul cohomology and property $(N_p)$ in general. In Section 3 we discuss property $(N_p)$ on surfaces,  and among others  we prove the equivalence of $(2)$ and $(3)$ of Theorem~A.
Section 4 is devoted to the proof of Theorem A and some consequences, note again that the results here usually provide estimates on $(L^2)$ that are considerably stronger than the ones in Theorem A. Section 5 hosts the proof of Theorem B. 

\subsection*{Acknowledgements}
We are grateful to Giuseppe Pareschi, Angelo Lopez, Klaus Hulek, Gavril Farkas, and Andreas Leopold Knutsen for helpful conversations. The first author was  supported by the grant IRTG 1800  of the DFG.

\section{Preliminaries} 

We work over the complex numbers, every variety is connected, smooth, and projective, unless otherwise mentioned. 
We recall that for a polarized  variety $(X,L)$ and an arbitrary point $x\in X$, the \textit{Seshadri constant} of $L$ at $x$ is defined to be
 \[
  \epsilon(L;x) \ \deq \ \inf_{x\in C\subseteq X} \frac{(L\cdot C)}{\textup{mult}_x(C)} \ .
 \]
If the infimum is achieved by some curve $C\subseteq X$, then we call $C$ a \textit{Seshadri exceptional curve}.

\subsection{Koszul cohomology}
Let $V$ be a vector space of dimension $n$ and $S=\operatorname{Sym}^{\bullet}V$ the graded symmetric algebra over $V$. Then for any finitely generated graded $S$-module $M$ we have an unique (up to isomorphism) \textit{minimal graded free resolution}:
\[ 0 \longrightarrow F_n \longrightarrow F_{n-1} \longrightarrow \dots \longrightarrow F_1 \longrightarrow F_0 \longrightarrow M \longrightarrow 0  \]
where the $F_p$ are graded free $S$-modules of finite rank. We can write  
\[ F_p = \bigoplus_{q\in\ZZ} K_{p,q}(M;V)\otimes_{\CC} S(-p-q) \]
for certain finite dimensional vector spaces $K_{p,q}(M;V)$ which are called the \textit{Koszul cohomology groups} of $M$ w.r.t $V$. The name Koszul cohomology comes from the fact that they can be computed as the middle cohomology of the Koszul complex \cite{G84}:
\[ \wedge^{p+1}V\otimes M_{q-1} \to \wedge^p V \otimes M_q \to \wedge^{p-1}V\otimes M_{q+1} \qquad v_1\wedge \dots \wedge v_p \otimes m \mapsto \sum_{i=1}^p(-1)^{i+1}v_1 \wedge \dots \wedge \widehat{v_i} \wedge \dots \wedge v_p \otimes v_i\cdot m \]

We will need later to compare Koszul cohomology with respect to  two different vector spaces. More specifically, suppose that we have a short exact sequence of vector spaces 
\[ 0\longrightarrow U \longrightarrow V \longrightarrow W \longrightarrow 0 \ ,\] 
and a finitely generated graded $\operatorname{Sym}^{\bullet}(W)$-module $M$. Then $M$ has also a $\operatorname{Sym}^{\bullet}(V)$-module structure, and we can compare the Koszul 
cohomologies computed with respect to $W$ and $V$. 

\begin{lemma}\label{lemma:koszvw}
In the above situation, we have a non-canonical isomorphism
\[ K_{p,q}(M;V) \cong \bigoplus_{i=0}^p \wedge^{p-i}U \otimes K_{i,q}(M;W) \]
In particular $K_{p,q}(M;W)\subseteq K_{p,q}(M;V)$.
\end{lemma}
\begin{proof}
Fix a splitting $V=U\oplus W$. Then $\wedge^p V = \bigoplus_{i=0}^p \wedge^{p-i}U \otimes \wedge^i W$ and the Koszul complex behaves well w.r.t this splitting: indeed, since $U\subseteq \operatorname{Ann}(M)$ we see that
\[ u_1 \wedge \dots \wedge u_{p-i} \wedge w_1 \wedge \dots \wedge w_i \otimes m \mapsto u_1\wedge \dots \wedge u_{p-i} \wedge \left( \sum_{k=0}^i (-1)^{k+1}\wedge w_1 \wedge \dots \wedge \widehat{w_k} \wedge \dots \wedge w_i \otimes w_k \cdot m \right)  \]
thus, the Koszul complex of $(M;V)$ splits and Koszul cohomology splits as well. The last statement follows by taking $i=p$.
\end{proof}

In the geometric setting, let $X$ be an irreducible projective variety of positive dimension, $\mathcal{F}$ a coherent sheaf on $X$ and $L$ a globally generated ample line bundle on $X$. We denote by $S\deq \Sym^{\bullet}H^0(X,L)$ the homogeneous coordinate ring of the space $\PP(H^0(X,L))$ and by $\Gamma_X(\mathcal{F},L) = \bigoplus_{q\in \mathbb{Z}}H^0(X,\mathcal{F}\otimes qL)$ the module of sections of $\mathcal{F}$. Then $\Gamma_X(\mathcal{F},L)$ 
has a natural structure of a finitely generated graded $S$-module so that we can take the Koszul cohomology groups
\begin{align*}
  K_{p,q}(X,\mathcal{F},L) & := K_{p,q}(\Gamma_X(\mathcal{F},L);H^0(X,L)) \\
  K_{p,q}(X,L) & := K_{p,q}(X,\mathcal{O}_X,L)                           
\end{align*}

Property $(N_p)$ for polarized varieties $(X,L)$ has been introduced by Green and Lazarsfeld in \cite{GL87}. 

\begin{definition}[Property $(N_p)$]
In the above situation, we say that $(X,L)$ has property $N_p$ if $K_{i,q}(X,L)=0$ for all $i\leq p$ and $q\geq 2$. 
\end{definition}

\begin{remark}
Property $(N_p)$ means that the resolution is as simple as possible in the first $p$ steps:
\[ 
\dots \to S(-p-1)^{\oplus k_{p}} \to S(-p)^{\oplus k_{p-1}} \to \dots \to S(-2)^{\oplus k_1} \to S \to \Gamma_X(\mathcal{O}_X,L) \to 0 
\]
with $k_{i} \in \mathbb{N}$. In particular, property $(N_0)$ coincides with projective normality, and property $(N_1)$ means that the embedding is projectively normal and that the homogeneous 
ideal is generated by quadrics.
\end{remark}

\subsection{Syzygies on curves}

Canonically polarized curves are an important and much studied case. In this situation, we expect that property $(N_p)$ is determined by  the Clifford index of the curve.

\begin{definition}[Clifford index]
  Let $C$ be a smooth, irreducible, projective curve of genus $g\geq 4$. We define the Clifford index of $C$ as
  \[ \operatorname{Cliff}(C) = \min\left\{ \deg(L) - 2h^0(X,L) + 2 \,|\, L \in \operatorname{Pic}(C), h^0(C,L)\geq 2, h^1(C,L)\geq 2 \right\} \]
  Moreover, if $C$ is a smooth curve of genus $2$ or $3$, we define the Clifford index of $C$ to be zero if the curve is hyperelliptic and one otherwise.
\end{definition}
% \begin{rmk} 
%   Let $\operatorname{gon}(C)$ denote the gonality of the curve $C$, that is, the smallest degree of a map $C\to \PP^1$. Then it has been proven by  Coppens and Martens \textbf{NEEDS REFERENCE} that
%   \[ \operatorname{Cliff}(C) \in \{ \operatorname{gon}(C)-2, \operatorname{gon}(C)-3 \} \]
% \end{rmk}
Generalizing some classical results about canonical curves,  Green \cite{G84} proposed the following.

\begin{conjecture}[Green]
Let $C$ be a smooth, irreducible, projective curve of genus $g\geq 2$. Then $(C,\omega_C)$ has property $N_p$ if and only if $p<\operatorname{Cliff}(C)$.
\end{conjecture}

This conjecture is open in general, but it has been proven for general curves by Voisin \cites{V02,V05} using K3 surfaces. 
Building on her work, Aprodu and Farkas  \cite{AF11} verified  the conjecture  for every smooth curve on any  K3 surface.

\section{Property $(N_p)$ on surfaces}

This section contains general results about property  $(N_p)$ on surfaces. Let $X$ be a smooth projective surface, $L$  an ample and globally generated line bundle on $X$. Then the evaluation map on global sections gives rise to the  exact sequence
\[
 0\longrightarrow \ M_L \ \longrightarrow \ H^0(X,L)\otimes \sO_X \ \longrightarrow \ L \ \longrightarrow 0 \ ,
\]
where $M_L$ is a locally free, it is called the  \textit{syzygy bundle} of $L$.

As pointed out in the introduction, we intend to characterize property $(N_p)$ on surfaces 
via the non-existence of certain low degree curves.  In one direction we  one can detect failure of property $N_p$ on surfaces by restricting to curves.

\begin{proposition}\label{prop:easier direction}
Let $X$ be a smooth projective surface, $L$ an ample and globally generated line bundle on $X$ and $p\geq 0$ be some nonnegative integer. Suppose that there is a reduced effective divisor $F \subseteq X$ of  arithmetic genus $p_a(F) \geq 1$ such that $(L\cdot F) \leq p+2$. Then $(X,L)$ does not have property $N_p$.
\end{proposition}
\begin{proof}
  We can clearly suppose $F$ to be connected, otherwise we can restrict to a subdivisor. First we claim that $h^0(F,\sO_F(L)) \leq (L\cdot F)$: by Riemann-Roch, we get that
  \[ h^0(\sO_F(L)) = (L\cdot F) + 1 - h^1(\sO_F) + h^1(\sO_F(L)) \]
  and recalling that $\sO_F(K_X+F)$ is  dualizing  on $F$, we can rewrite this as
  \[ h^0(\sO_F(L)) = (L\cdot F) +1 - h^0(\sO_F(K_X+F)) + h^0(\sO_F(K_X+F-L)) \]
  so that the claim is equivalent to $h^0(\sO_F(K_X+F-L)) \leq  h^0(\sO_F(K_X+F)) -1$. To show this, we observe that $h^0(\sO_F(K_X+F)) = p_a(F) \geq 1$, so that there exists a nonzero section in  $ H^0(F,\sO_F(K_X+F))$: since $F$ is reduced, there exists a point $p\in F$ where the section does not vanish, and since $L$ is globally generated, we can find a curve $C \in |L|$ that intersects $F$ transversally and passes through $p$. Taking cohomology in the short exact sequence
  \[ 0 \to \sO_F(K_X+F-L) \to \sO_F(K_X+F) \to \sO_{C\cap F}(K_X+F) \to 0 \]
  we get an exact sequence
  \[ 0 \to H^0(\sO_D(K_X+F-L)) \to H^0(\sO_F(K_X+F)) \to H^0(\sO_{C\cap F}(K_X+F)) \]
  and since the last map is nonzero by construction we see that $h^0(\sO_F(K_X+F-L)) \leq h^0(\sO_F(K_X+F))-1$.

  The rest of the proof goes as in the proof of \cite{KL15}{Theorem 4.1}: suppose now that $(X,L)$ has property $N_p$, then in particular $L$ is very ample, so that we can consider $X$ as an embedded surface $X\subseteq \mathbb{P}^N$. Let $\Lambda$ be the linear space spanned by $F$. By definition $\Lambda$ is the projectivization of the image of the restriction map $H^0(X,L)\to H^0(F,\sO_F(L))$, hence
  \[ \dim \Lambda \leq h^0(F,\sO_F(L)) -1 \leq (L\cdot F) -1 \leq p+1 \]
By sheafifying the minimal free resolution of $\Gamma_X(\mathcal{O}_X,L)$ we get a resolution of sheaves on $\mathbb{P}^N$:
\[ 0 \to E_n  \to E_{n-1} \to \dots \to E_1 \to E_0 \to \sO_X \to 0 \qquad E_p = \bigoplus_q K_{p,q}(X,L)\otimes \sO_{\mathbb{P}^N}(-p-q)   \]
and  if we restrict this resolution to $\Lambda$, we get a complex
\[ 0 \to {E_n}_{|\Lambda}  \to {E_{n-1}}_{|\Lambda} \to \dots \to {E_1}_{|\Lambda} \to {E_0}_{|\Lambda} \to \sO_{X\cap \Lambda} \to 0 \qquad {E_p}_{|\Lambda} = \bigoplus_q K_{p,q}(X,L)\otimes \sO_{\Lambda}(-p-q)   \]
which is exact at $\sO_{X\cap \Lambda}$ and everywhere exact outside $X\cap \Lambda$, which has dimension at most one, as $X\subseteq \PP^N$ is nondegenerate. Since $(X,L)$ has property $N_p$, we know that $ K_{i,q}(X,L)=0$ for all $ 0 \leq i \leq p,\,\, q\geq 2$ and in particular we obtain that $H^i(\Lambda, {E_i}_{|\Lambda})=0$ for all $i>0$ and $0\leq p \leq \dim \Lambda -1$. Hence using \cite{GLP}*{Lemma 1.6} we get that $H^1(X\cap \Lambda, \sO_{X\cap \Lambda})=0$. Since $F$ is a subscheme of $X\cap \Lambda$, we have an exact sequence
\[ 0 \to \mathcal{I}_{F/X\cap \Lambda} \to \sO_{X\cap \Lambda} \to \sO_F \to 0 \]
and taking cohomology we see that $H^2(X\cap \Lambda,\mathcal{I}_{F/X\cap \Lambda})=0$, because $X\cap \Lambda$ has dimension one. Then,
it follows that $H^1(F,\mathcal{O}_F)=0$, which is absurd because $p_a(F)\geq 1$.
\end{proof}

% \begin{framed}We will exploit this fact using an extension result of Serrano \cite[Theorem 3.1]{Se87} for morphisms defined on smooth divisors inside surfaces:
% 
% \begin{theorem}[Serrano]\label{thm:serrano}
% Let $X$ be a smooth, irreducible projective surface and $C\subseteq X$ a smooth, irreducible curve. Let also $f\colon C \longrightarrow \PP^1$ be a morphism of degree $d\geq 2$ and suppose that $(C^2) > (d+1)^2$ (or that $(C^2) > \frac{1}{2}(d+2)^2$ and $K_X$ is numerically even). Then there is a morphism $\overline{f}\colon X \longrightarrow \PP^1$ which extends $d$. 
% \end{theorem}
% \end{framed}

The following is a general condition under which  property $(N_p)$ for a polarized surface $(X,L)$ is inherited by curves in $X$.

\begin{proposition}\label{lem:surfacetocurve}
Let $X$ be a surface,  $L$ an ample and globally generated line bundle on $X$. Let $F\subseteq X$ be a smooth curve such that
\begin{enumerate}
\item $H^1(X,L-F)=H^1(X,2L-F)=0$.
\item $H^2(X,L-F)=0$.
\item $K_{i,0}(X,K_X+F,L)=H^0(X,\wedge^{i}M_L \otimes(K_X+F))=0$ for all $i=h^0(X,L)-p-2$.
\end{enumerate}
If $(X,L)$ has property $(N_p)$, then so does  $(F,\sO_F(L))$.
\end{proposition}

\begin{proof}
Observe first that  conditions $(1)$ and $(2)$ together  imply that $\sOX(-F)$ is $3$-regular with respect to $L$ in the sense of Castelnuovo-Mumford, it follows in particular that 
$H^1(X,qL-F)=H^2(X,qL-F)=0$ for all $q\geq 1$. Assume that $(X,L)$ satisfies property $N_p$, and consider the short exact sequence
\[ 
0 \longrightarrow \sOX(-F) \longrightarrow \sOX \longrightarrow \sO_F \longrightarrow 0 \ . 
\]
This induces to an exact sequence of  graded $\operatorname{Sym}^{\bullet} H^0(X,L)$-modules
\[ 
  0 \longrightarrow \Gamma_X(\sO_X(-F),L) \longrightarrow \Gamma_X(\sO_X,L) \longrightarrow M \longrightarrow 0
\]
where $M$ is a graded module such that $M_q = H^0(F,\sO_F(qL))$ for all $q\geq 1$. In particular, from the description of the $K_{p,q}$ via the Koszul complex, we see that $K_{j,q}(F,\sO_F(L);H^0(X,L)) = K_{j,q}(M;H^0(X,L))$ for all $q\geq 2$, so that, if we can prove $K_{j,q}(M;H^0(X,L))=0$ for all $j\leq p, q\geq 2$, we are done thanks to Lemma \ref{lemma:koszvw}.
Using the long exact sequence in Koszul cohomology \cite{G84}*{Corollary 1.d.4}, arising from the above sequence and the assumption that $(X,L)$ has property $(N_p)$, we see that $K_{j,q}(M;H^0(X,L)) \subseteq K_{j-1,q+1}(X,-F,L)$ and we claim that this latter group is zero. When $q\geq 3$ this follows from the regularity of $\sOX(-F)$ with respect to $L$ \cite{AN10}*{Proposition 2.37}. If $q=2$ we can use Green's duality theorem for Koszul cohomology (see  \cite{G84}*{Theorem~2.c.6}), which gives 
\[
 K_{j-1,3}(X,-F,L)^{\vee} \equ  K_{h^0(X,L)-j-2,0}(X,K_X+F,L) \ ,
\]
To conclude, we observe that $H^0(K_X+F-L)=H^2(L-F)=0$ thanks to condition (2), and looking at the Koszul complex we immediately see that $K_{j,q}(X,K_X+F,L)=0$ for all $j\leq 0$. Then \cite{eisenbud}*{Proposition 1.9} shows that condition (3) implies the vanishing $K_{h^0(X,L)-j-2,0}(X,K_X+F,L)$ for all $j\leq p$. 
\end{proof}

\begin{remark}\label{remark:stableML} 
To check condition $(3)$ we can sometimes use the stability considerations for the bundle $M_L$. 
Assuming that $M_L$ is slope stable with respect to $L$, the vanishing that we want is implied by
\[
  \left( c_1(\wedge^{h^0(X,L)-p-2}M_L\otimes (K_X+F))\cdot L  \right) \ < \ 0 \ .
\]
If we denote by $r=h^0(X,L)-1$ the rank of $M_L$ and $i=h^0(X,L)-p-2$, then the intersection number on the left is computed as follows:
\begin{align*} 
c_1(\wedge^{i}M_L\otimes (K_X+F))\cdot L & = c_1(\wedge^{i}M_L)\cdot L + \binom{r}{i}((K_X+F)\cdot L) 
= \binom{r-1}{i-1}c_1(M_L)\cdot L + \binom{r}{i}((K_X+F)\cdot L) \\
&  = -\binom{r-1}{i-1}(L^2) + \binom{r}{i}((K_X+F)\cdot L) = -\binom{r-1}{i-1}\left[ (L^2) - \frac{r}{i}((K_X+F)\cdot L)  \right] \\
& = -\binom{r-1}{i-1}\left[ (L^2) - \frac{h^0(X,L)-1}{h^0(X,L)-p-2}((K_X+F)\cdot L) \right]   
\end{align*}
In particular, condition $(3)$ in Proposition~\ref{lem:surfacetocurve} is satisfied, if $M_L$ is slope stable with respect to  $L$, and 
\[
 (h^0(X,L)-p-2)\cdot (L^2) \ > \ (h^0(X,L)-1)\cdot ((K_X+F)\cdot L)\ .
 \]
\end{remark}

As an illustration of this technique, we prove the following extension of Proposition \ref{prop:easier direction}:

\begin{proposition}\label{thm:one direction}
Let $X$ be a K3 surface, an abelian surface or an Enriques surface, and $L$ an ample and globally generated line bundle such that $(L^2)\geq 3p+6$ (if K3), $(L^2)\geq 3p+14$ (if abelian) or $(L^2)\geq 3p+7$ (if Enriques). Suppose there exists a smooth curve $F\subseteq X$ of genus two with $(L\cdot F)\leq p+4$ %($\Leftrightarrow$ exists a generically $2:1$ map $X\rightarrow \PP^2$), 
, then  $(X,L)$ does not satisfy property $(N_p)$.
\end{proposition}

\begin{proof}   
  Suppose that $(X,L)$ satisfies property $(N_p)$, in particular $L$ is very ample.  We start with computing  $h^0(F,\sO_F(L))$. Since $\sO_F(L)$ is very ample, we must have  $(L\cdot F) \geq 3$, so that Riemann--Roch gives $h^0(F,\sO_F(L)) = (F\cdot L) -1 \leq p + 3$.
Since $L$ is ample, we have $H^1(X,L)=0$ by Kodaira vanishing. Consider then the exact sequence
  \[ 
  0 \to H^0(X,L-F) \to H^0(X,L) \to H^0(F,\sO_F(L)) \to H^1(X,L-F) \to 0 \ .
  \]
  We see that if $H^1(X,L-F)\ne 0$, then the image of the restriction map on global sections has dimension at most $h^0(F,sO_F(L))-1\leq p+2$, hence we are done using the same proof as in Proposition \ref{prop:easier direction}. Therefore, we can assume without loss of generality that $H^1(X,L-F)= 0$. 
 Now we try to check the conditions of Lemma \ref{lem:surfacetocurve}: if $C\in|L|$ is a smooth curve then we can look at the short exact sequence
\[
   0\to \sO_X(L-F)\to\sO_X(2L-F)\to\sO_C(2L-F)\to 0\ .
\]
We see that $((2L-F)\cdot L) > (L^2)$, so that Riemann-Roch implies $H^1(C,\sO_C(2L-F))=0$, and then $H^1(X,L-F)= 0$ gives $H^1(X,2L-F)= 0$ as well. Serre duality and $(L\cdot ( L-F))>0$ imply that we have $H^2(X,L-F)=0$.

Thus, we are left with checking that condition $(3)$ of  Lemma~\ref{lem:surfacetocurve} holds. One way to verify this according to  Remark~\ref{remark:stableML} is to  show that 
the syzygy bundle $M_L$  is stable. Note that under the conditions in the statement we have the inequality
\[
 (h^0(X,L)-p-2)\cdot (L^2) \ > \ (h^0(X,L)-1)\cdot (F\cdot L)\ ,
 \]
The stability condition on $M_L$ was proved  in \cite{Camere2012}, hence, by Lemma~\ref{lem:surfacetocurve}  $(F,\sO_F(L))$ has property $(N_p)$. 
However a line bundle of degree $d\leq p+4$ on a hyperelliptic curve of genus at least $2$ can never satisfy $(N_p)$ by \cite{GL88}*{Theorem~2}.
\end{proof}

Finally, we establish the equivalence between $(2)$ and $(3)$ of Theorem~A on abelian and K3 surfaces under a light numerical condition. 

\begin{proposition}\label{prop:seshadri}
Let  $X$ be a surface with $K_X=0$ and $L$  an ample line bundle on $X$ with $(L^2)>\frac{8}{7}(p+2)^2$  for a given  natural number  $p$. 
Then conditions $(2)$ and $(3)$ in Theorem~A are equivalent.
\end{proposition}

\begin{proof}
The implication $(3)\Longrightarrow (2)$ follows immediately from the definition of the Seshadri constant; therefore,  it remains to deal with the converse implication.  Let $x\in X$ be a very general point. If the Seshadri constant is maximal, i.e. $\epsilon(L;x)=\sqrt{(L^2)}$, then $(3)$ follows from  the given lower bound  on $(L^2)$.
Assume that $\epsilon(L;x)$ is submaximal, then there exists a curve $F\subseteq X$ with $m\deq \mult_x(F)\geq 1$ giving the Seshadri constant $\epsilon(L;x)$. 
If $m=1$ and $(L\cdot F)\leq p+2$, then the Hodge index theorem yields  $(F^2)\leq 0$. However, since $x\in X$ was chosen to be very general, it does not lie  on any negative curve,
in particular we obtain  that $(F^2)=0$. Now, we can assume $F$ to be irreducible and reduced, and if $X$ is an abelian surface $F$ must be also smooth, since there are no rational curves on $X$. If instead $X$ is a K3 surface, we can again suppose that $F$ is smooth via \cite{beauville}*{Proposition VIII.13}.  In both cases, we get to a contradiction of condition $(2)$, therefore,  either $m=1$ and $(L\cdot F)\geq p+3$, or $m\geq 2$.
 
 In the first case condition  $(3)$ is immediate, so we can assume that $m\geq 2$. First, note that $X$ is not uniruled, and then \cite{KSS}*{Theorem 2.1} implies that
 \[
  (F^2)  \dgeq  m^2 -m +2\ .
 \]
 Next, the definition of the Seshadri constant and the Hodge index theorem imply the  sequence of inequalities
 \[
  \epsilon(L;x)^2 \equ  \frac{(L\cdot F)^2}{m^2}  \dgeq  \frac{(L^2) \cdot (F^2)}{m^2} \, > \, \frac{8}{7}(p+2)^2\left(\frac{m^2-m+2}{m^2}\right) \dgeq (p+2)^2 \ ,
 \]
 where the last inequality is valid for any $m\geq 1$.  In particular, $\epsilon(L;x)> p+2$, which yields $(3)$.
\end{proof}

\section{Higher syzygies on K3 surfaces}

We devote this section to the proof of Theorem A.  First  we  present a stronger variant of the equivalence between $(1)$ and $(2)$ in Theorem A for K3 surfaces. The fundamental principle here is that Koszul cohomology for K3 surfaces can be reduced to curves:

\begin{lemma}\label{lem:Koszul}
	Let $(X,L)$ be a polarized K3 surface, with $L$ ample and globally generated. Let $C\in |L|$ be any smooth curve, then
	\[K_{p,q}(X,L) \cong K_{p,q}(C,\omega_C) \qquad \text{ for all } p,q  \]  
\end{lemma}
\begin{proof}
	This follows from a general Lefschetz theorem for Koszul cohomology (see \cite{AN10}*{Theorem 2.20}) together with the fact that $H^1(X,L^q)=0$ for all $q\geq 0$, since we are on a K3 surface.
\end{proof}

Moreover, we know from results of Aprodu-Farkas that Green's conjecure holds for every smooth curve on any K3 surface, so that we have:

\begin{corollary}[Aprodu--Farkas]\label{cor:AF}
	Let $(X,L)$ be a polarized K3 surface with $L$ ample and globally generated. Then $(X,L)$ has property $N_p$ if and only if for one (or any) smooth curve $C\in |L|$ one has $p<\operatorname{Cliff}(C)$. 
\end{corollary}

Then, it becomes important to study the Clifford index of curves on K3 surfaces. For example, we know that for a smooth curve $C$ we have $\operatorname{Cliff}(C)\in \{ \operatorname{gon}(C)-2,\operatorname{gon}(C)-3 \}$. In the K3 case one can make a very precise distinction (see \cite{Kn01}*{Proposition 8.6}):

\begin{proposition}
	Let $(X,L)$ be a K3 surface, with $L$ big and globally generated. Suppose that for all smooth curves $C\in |L|$ one has $\operatorname{Cliff}(C)=\operatorname{gon}(C)-3$.  Then $L\cong \sOX(2D+\Gamma)$, where $D\subseteq X$ is an effective divisor  and $\Gamma\subseteq X$ is a smooth curve such that:
	\[ (D^2) \geq 1 \qquad (\Gamma^2)=-2 \qquad (\Gamma\cdot D) = 1 \]
	In this case, for any smooth curve $C\in |L|$ we have
	\[ \operatorname{Cliff}(C) = (D^2)-1 \qquad \operatorname{gon}(C) = (D^2)+2 \]   
\end{proposition}

\begin{remark}\label{notample}
In particular, we see that the line bundle $L$ described in the above proposition is not ample, since $(L\cdot \Gamma) = 2(D\cdot \Gamma) + (\Gamma^2) = 0$. %Hence, if $L$ is an ample and globally generated line bundle on a K3 surface, the general smooth curve $C\in |L|$ has $\operatorname{Cliff}(C) = \operatorname{gon}(C)-2$.
\end{remark}

\begin{theorem}\label{thm:gle}
	Let $X$ be a K3 surface and $L$ be an ample line bundle on $X$. Then
	\begin{enumerate}
		\item[(a)] If $(L^2)>\frac{1}{2}(p+4)^2$ then  conditions $(1)$ and $(2)$ in Theorem A are equivalent.
		\item[(b)] If $(L^2)> \frac{1}{4}(p+6)^2$, then the conditions $(1)$ and $(2)$ in 
		Theorem A are equivalent with the exception of the case when  there exists a smooth genus two curve $F$ on $X$ such thata $1\leq (L\cdot F)\leq p+4$.
	\end{enumerate}
\end{theorem}

\begin{remark}
	If $f:X\rightarrow\PP^2$ is a cover of degree $2$ branched over a generic sextic, then $\Pic(X)\simeq \Pic(\PP^2)$, by \cite{Bu}. 
	In particular, there is no elliptic fibrations on $X$ and condition $(2)$ in theorem A holds  for $L=f^*(O_{\PP^2}(m+2))$. However, by Proposition~\ref{prop:easier direction},  
	the pair $(X,L)$ satisfies $(N_{2m-1})$ but not $(N_{2m})$. 
\end{remark}

\begin{proof}[Proof of Theorem~\ref{thm:gle}]
	In both cases the implication $(1)\Longrightarrow (2)$ in Theorem~A is proved under less restrictive hypotheses in  Proposition~\ref{prop:easier direction}. Now we address the converse implication.

        $(a)$ Suppose that $(X,L)$ does not have property $N_p$. Since $L$ is ample, we know from Remark \ref{notample} that there exists a smooth curve $C\in |L|$ with  $\operatorname{Cliff}(C) = \operatorname{gon}(C)-2$. Then Corollary \ref{cor:AF} implies that $\operatorname{Cliff}(C)\leq p$ which is the same as  $\operatorname{gon}(C) \leq p+2$. Now, we could proceed in various ways: for example it is easy to conclude the proof using a  result  of Green--Lazarsfeld from  \cite{GL86}, or Knutsen's  \cite{Kn01}*{Theorem 1.3}, but we would like to use an extension result due to Serrano, since in principle it could be applied also on other surfaces, whereas the other results we mentioned are specific to the K3 case.

  So, by definition of gonality, there exists a map $f\colon C \to \mathbb{P}^1$ of degree $\deg f\leq p+2$. Then \cite{Se87}*{Theorem 3.1} implies that $f$ extends to a morphism $\overline{f}:X\rightarrow \PP^1$: the general fiber $F$ of $\overline{f}$ is a smooth elliptic curve and by construction we have that $(L\cdot F) = (C\cdot F) = \operatorname{deg} f \leq p+2$.

  $(b)$ As before, if $(X,L)$ does not have property $N_p$, then there is a smooth curve $C\in |L|$ with a map $f\colon C \rightarrow \PP^1$ with degree $d=\operatorname{gon}(C) \leq p+2$. The difference from case $(a)$ is that due to the weaker numerical constraint, \cite{Se87}*{Theorem 3.1} does not apply. Instead, we will make use of another result by Serrano \cite{Se87} that requires some extra work.
        Let $x\in\PP^1$ be a general point, then  the fiber $f^{-1}(x)=\{P_1^x, \ldots , P_d^x\}$ consists of $d$ 
	distinct points. In the course of the proof of \cite{Se87}, Serrano shows that there exists a non-trivial effective divisor $F_x$ and a $\QQ$-effective $P_x$ on $X$ with $C=F_x+P_x$ 
	satisfying
	\[
	0\leq a\deq (F_x^2) \ <\ b\deq (P_x^2), \qquad a \dleq d \dleq (p+2), \qquad  0<e\deq (F_x\cdot P_x) \leq d \ ,
	\]
	and  $P^x_1,\ldots ,P^x_e\subseteq C\cap F_x$. Moreover,	since  $H^1(X,\sO_X)=0$, it follows from  \cite{Se87}*{Theorem~3.9} that the effective divisor $F_x$ can be chosen to be irreducible and that infinitely many of them are in the same linear equivalence class, as $x\in \PP^1$ vare
	and to remain  in the same linear series while  $x\in\PP^1$ varies. In particular $F_x$ is globally generated by \cite{SD}*{Theorem~ 3.1}, and we can assume that $F_x$ is smooth. 

        Then, if $a=0$ we see that $F_x$ is a smooth elliptic curve such that $(F_x \cdot L) = (F_x^2)+(F_X\cdot P_x) \leq p+2$. If instead $a>0$, then the Hodge index theorem implies that $e^2 \geq ab$, thus by our hypotheses we get
	\[
  \frac{1}{4}(p+6)^2 <	(C^2) \equ a+b+2e  \dleq a+\frac{e^2}{a} + 2e \dleq a+\frac{d^2}{a}+2d \dleq a+\frac{(p+2)^2}{a}+2(p+2)  .
	\]
Now we observe that $a\leq d \leq (p+2)$ and  the function $f(a)=a+(p+2)^2/a+2(p+2)$ is decreasing for $0<a\leq (p+2)$. Then, the above inequality implies $a\leq 3$ and since the intersection form on a K3 surface is even it must be $(F_x^2)=2$. Then $F_x$ is a smooth genus two curve such that $(C\cdot F_x)\leq p+4$.
\end{proof}

Using this result it is straightforward to give a characterization of property $(N_p)$ for ample line bundles of type $L^{\otimes m}$ with $m\geq p$, in the spirit of Mukai's conjecture. In particular this gives an  extension of results of Gallego and Purnaprajina \cite{GP00}.

\begin{corollary}
  Let $X$ be a K3 surface and $L$ an ample line bundle, then $(p+3)\cdot L$ has property $(N_p)$.
\end{corollary}

\begin{proof}
 Let $M=(p+3)\cdot L $: then $M$ is globally generated \cite{SD} and $(M^2)=(p+3)^2(L^2)$ satisfies the condition of Theorem \ref{thm:gle} (a). Hence $M$ fails to have property $(N_p)$ if and only if there exists a smooth elliptic curve $E\subseteq X$ such that $(E\cdot M) \leq p+2$, but $(E\cdot M) = (p+3)(L\cdot E) \geq p+3$.
\end{proof}

For simplicity of exposition, in the next cases we  restrict ourselves to the case when $p\geq 2$, since properties $(N_0)$ and $(N_1)$ are taken care of by Saint-Donat's Theorem \cite{SD}.

\begin{corollary}
  Let $X$ be a K3 surface, $L$ an ample line bundle and $p\geq 2$ an integer. Then $(p+2)\cdot L$ has property $(N_p)$ if and only if there is no smooth elliptic curve $E\subseteq X$ such that $(L\cdot E)=1$.
  %In the exceptional case, if $(L^2)=2$ then $2\cdot L$ does not have property $(N_0)$.
\end{corollary}

\begin{proof}
Let $M=(p+2)\cdot L$: then $M$ is globally generated \cite{SD} and we see that the self-intersection $(M^2)=(p+2)^2(L^2)$ satisfies the condition of  Theorem \ref{thm:gle} (a). Hence $M$ fails to have property $(N_p)$ if and only if there exists an elliptic curve $E\subseteq X$ such that $(E\cdot M) \leq p+2$, which is equivalent to $(E\cdot L)=1$.
%The exceptional case follows from Saint-Donat's theorem \ref{thm:saintdonat}.
\end{proof}

\begin{corollary}
  Let $X$ be a K3 surface, $L$ an ample line bundle and $p\geq 2$ an integer. If $(L^2)\geq 4$, or $(L^2)=2$ and $p>2$, then $(p+1)\cdot L$ has property $(N_p)$ if and only if there is no smooth elliptic curve $E\subseteq X$ such that $(E\cdot L)=1$.
 Moreover, if $(L^2)=2$, then $3\cdot L$ does not have property $(N_2)$.
\end{corollary}

\begin{proof}
 Suppose $(L^2)\geq 4$ or $(L^2)=2$ and $p>2$. Then $M=(p+1)\cdot L$ is globally generated \cite{SD}, and $(M^2)$ satisfies the condition of Theorem \ref{thm:gle} (a). Hence $M$ fails property $(N_p)$ if and only if there is a smooth elliptic curve $E\subseteq X$ such that $(M\cdot E) \leq p+2$. But since $p>0$, this is equivalent to $(L\cdot E)=1$.
  Now, suppose that $(L^2)=2$:  we need to show that $3\cdot L$ does not have property $N_2$. Thanks to Lemma \ref{lem:Koszul} and Corollary \ref{cor:AF} this is equivalent to showing that for a smooth curve $C\in |3\cdot L|$ we have $\operatorname{Cliff}(C)\leq 2$. So, let $C$ be such a curve and consider the exact sequence
    \[ 0 \to \sO_X(-2L) \to \sO_X(L) \to \sO_C(L) \to 0. \]
  Taking cohomology we see that $h^0(C,\sO_C(L)) = h^0(X,L) =3$ and $h^1(C,\sO_C(L)) = h^2(X,2L) = 6$, so that $\operatorname{Cliff}(C) \leq (C\cdot L) - 2h^0(C,\sO_C(L)) + 2 = 3(L^2) -4 = 6-4 =2$.   
\end{proof}

\begin{corollary}
  Let $X$ be a K3 surface and $L$ an ample line bundle and $p\geq 2$ an integer.
  \begin{enumerate}
  \item if $(L^2)\geq 4$ and $p>2$, or $(L^2)=2$ and $p>4$, then $p\cdot L$ has property $(N_p)$ if and only if there is no smooth elliptic curve $E\subseteq X$ such that $(E\cdot L)=1$.
  \item if $(L^2)\geq 6$ then $2\cdot L$ has property $(N_2)$ if and only if there is no smooth elliptic curve $E\subseteq X$ such that $(E\cdot L)\leq 2$.
  \item in the other cases $p\cdot L$ does not have property $(N_p)$.
  \end{enumerate}
\end{corollary}

\begin{proof}
  \begin{enumerate}
  \item Set $M=p\cdot L$. Then $M$ is globally generated \cite{SD} and $(M^2)$ satisfies the condition of Theorem \ref{thm:gle} (a). Hence $M$ fails property $N_p$ if and only if there is a smooth elliptic curve $E\subseteq X$ such that $(E\cdot M) \leq p+2$. But since $p>2$, this is equivalent to $(E\cdot L)=1$.
  \item Set $M=2\cdot L$. Then $M$ is globally generated \cite{SD} and $(M^2)$ satisfies the condition of Theorem \ref{thm:gle} (a). Hence $M$ fails property $N_2$ if and only if there is a smooth elliptic curve $E\subseteq X$ such that $(E\cdot M) \leq 4$, which is equivalent to $(E\cdot L)\leq 2$.
  \item The remaining cases are when $(L^2)=4$ and $p=2$ or when $(L^2)=2$ and $p=2,3,4$. In all these cases, we can show that $p\cdot L$ does not have property $(N_p)$ reasoning as in the previous Corollary.
  \end{enumerate}
\end{proof}

\begin{proof}[Proof of Theorem A]
The proof is an immediate consequence of Proposition~\ref{prop:seshadri}, Theorem~\ref{thm:one direction}, and 
Theorem~\ref{thm:gle} in the K3, and \cite{KL15}*{Theorem 1.1} in the case of abelian surfaces.
\end{proof}

\section{Enriques surfaces}

We give here a result analogous to Theorem~A on Enriques surfaces. To do this, we will  assume  the Green--Lazarsfeld secant conjecture for certain curves on Enriques surfaces, which we state for the sake of reference. 

\begin{conjecture}[Green--Lazarsfeld secant conjecture]\label{conj:GL secant}
Let $C$ be a smooth curve of genus $g$, $L$ a line bundle of degree $\deg L \geq 2g+1+p -2h^1(C,L)-\Cliff(C)$. Then $L$ has property $(N_p)$ if and only if it is $(p+1)$-very ample.
\end{conjecture}

On  Enriques surfaces a result of Knutsen and Lopez \cite{knutsenlopez} gives us some control over the Clifford index. In order to state their result, we introduce some notation. Let $L$ be a globally generated 
line bundle on $X$, and set 
\begin{align*}
\phi(L) & \deq  \inf\{ |(F\cdot L)| \,\, \mid \,\, F\in \Pic (X), \,\, (F^2)=0, \,\, F \not\equiv 0 \} \\
\mu(L) & \deq  \inf\{ (B\cdot L)-2 \,\, \mid \,\, B\in \Pic (X), \,\, (B^2)=4, \,\, \phi(B)=2, \,\, B \not\equiv L \}\ .
\end{align*}

\begin{theorem}[\cite{knutsenlopez}, Corollary 1.2]\label{thm:KnutLop}
Let $L$ be a globally generated line bundle on an Enriques surface $X$ and let $C$ be a general curve in $|L|$. Then
\[ 
\Cliff(C) \equ  \min\left\{ 2\phi(L)-2, \mu(L)-2, \left\lfloor \frac{(L^2)}{4} \right\rfloor \right\} 
\]
\end{theorem}

\begin{remark}
In the definition  of $\phi(L)$ we can actually suppose that $F$ is effective. Indeed, if $F$ is a divisor with $(F^2)=0$ then Riemann-Roch gives $h^0(F) - h^1(F) + h^0(K_X - F) = 1$, 
hence if $F$ is not effective then $K_X-F$ is effective, but then $|((K_X-F)\cdot L)| = |-(F\cdot L)| = |(F\cdot L)|$.
\end{remark}

\begin{theorem}\label{thm:Enriques}
Let $X$ be an Enriques surface and $L$ an ample and globally generated line bundle on $X$ such that $(L^2)>4(p+2)^2$.
Assume  that Conjecture~\ref{conj:GL secant} holds for a general curve in  $|L|$. Then the following are equivalent
\begin{enumerate}
\item $L$ has property $N_p$.
\item there is no effective and reduced divisor $F$ with $(F^2)=0$ such that $1\leq (F\cdot L)\leq p+2$.
\end{enumerate}  
\end{theorem}

\begin{proof}
Assume first that there exists  an effective and reduced divisor $F$ with $(F^2)=0$ such that $1\leq (F\cdot L)\leq p+2$. Then we know that $L$ does not have  property $(N_p)$ by Proposition~\ref{lem:surfacetocurve}.

Conversely, suppose that there is no reduced effective divisor $F$ with $(F^2)=0$ and $1\leq (F\cdot L)\leq p+2$ and let $C$ be a general smooth curve in $|L|$. It is easy to see that the restriction maps
$H^0(X,qL) \to H^0(C,\sO_C(qL))$ are surjective for all $q\geq 0$, and then it follows  as in Lemma \ref{lem:Koszul} that $(X,L)$ has property $(N_p)$ if and only if $(C,\sO_C(L))$ 
does. We will assume the  Green--Lazarsfeld secant conjecture on $C$.

First we show that the line bundle $\sO_C(L)$ satisfies Green--Lazarsfeld condition on the degree. It is easy to check that $\sO_C(L) = \sO_C(K_C + \eta)$, where $\eta$ is torsion bundle of orderexactly two. In particular, if we write  $g$ for the genus of $C$, we have $\deg(\sO_C(L))=2g-2$ and $h^1(C,\sO_C(L))=0$. Hence, the Green--Lazarsfeld condition is equivalent to $\Cliff(C) \geq p+3$. 

Suppose now that $\Cliff(C) \leq p+2$. Since $(L^2) > 4(p+2)^2$, we observe that $\lfloor \frac{(L^2)}{4} \rfloor \geq (p+2)^2 > (p+2)$, hence, by  Theorem~\ref{thm:KnutLop}, 
there must exist either an effective and nontrivial divisor $F$ with $(F^2) = 0$ and $(F\cdot L) \leq \frac{p+4}{2} $, or a divisor $B$ with $(B^2)=4$ and $(L\cdot B) \leq p+4$. 
However, the second possibility does not happen  because if $B$ is a divisor with $(B^2) =4$ then
\[ 
(B\cdot L) \geq \sqrt{(B^2)} \sqrt{(L^2)} > 4 (p+2) \geq p+4 \ . 
\]
For the first possibility, suppose that there is an effective divisor such that $(F^2)=0$ and $(F\cdot L) \leq \frac{p+4}{2}$. Observe that $\frac{p+4}{2} \leq (p+2)$, so that if we can take 
$F$ to be reduced, we are done. First we observe that for every effective divisor $F'\leq F$, it must be that $(F'^2)\leq 0$: indeed, if $(F'^2)\geq 1$, then
\[
\left( \frac{p+4}{2} \right)^2 \geq (L\cdot F)^2 \geq (L\cdot F')^2 \geq (L^2)(F'^2) > 4(p+2)^2 \ ,
\]
which is impossible. Hence, if $F$ has a reduced subdivisor $F'\leq F$  with $(F'^2)\geq 0$, we are done. 

If this does not happen, then every connected component of the support of $F$ is a 
tree of smooth rational curves. We can clearly reduce to the case when the support is connected, so that we can write $F=n_1C_1+\dots+n_rC_r$ where the $n_i$ are positive integers and the $C_i$
are smooth rational curves with $(C_i^2)=-2$, and, if $j \ne i$ then $(C_i\cdot C_j) = 1$ if $j=i-1,i+2$, and $(C_i\cdot C_j)=0$ otherwise. Then we see that
\[ 
-(F^2) \equ  \sum_{i=1}^r 2n_i^2 - 2n_1n_2 - 2n_2n_3 + \dots - 2n_{r-1}n_r = n_1^2 + \sum_{i=1}^{r-1}(n_i-n_{i+1})^2 + n_r^2 > 0 \ ,
\]
so that $(F^2)= 0$ cannot happen. 

Next, we show that $\sO_C(L)$ is $(p+1)$-very ample. Suppose that it is not,  and let $Z\subseteq C$ be a zero-dimensional subscheme of minimum length $ \ell \leq p+2$ such that 
$H^0(C,\sO_C(L)) \to H^0(C,\sO_C(L)\otimes \sO_Z)$ is not surjective. In particular, since $L$ is globally generated, we must have  $\ell\geq 2$. Then, by \cite{Kn01}*{Proposition 3.7}, 
there exists an effective divisor $D$ on $X$, containing $Z$ for which  $(L\cdot D) \leq (D^2)+ \ell \leq 2\ell$. 

However since $D$ contains $Z$, we see that $(L\cdot D)=(C\cdot D) \geq \ell$, hence $(D^2)\geq 0$. Suppose that $(D^2)\geq 1$. Then we have
\[ 
4(p+2)^2 \geq 4\ell^2 \geq (L\cdot D)^2 \geq (L^2)(D^2) > 4(p+2)^2 \ ,
\]
which is impossible, so it must be $(D^2)=0$ and consequently,  $(L\cdot D) \leq \ell \leq p+2$. 

\end{proof}

\end{document}